\documentclass[11pt]{amsart}
\usepackage{amssymb,latexsym,comment,url,amsmath}

\usepackage[T1]{fontenc}
\usepackage{rotating,graphicx}
\usepackage{currvita}
\usepackage{diagbox} 
\usepackage{xcolor}
\usepackage{subcaption}
\usepackage{multirow}
\usepackage{array}
\usepackage{longtable}
\usepackage{ltablex}
\usepackage{float}
\restylefloat{table}

\newcommand{\p}{{\mathfrak{p}}}
\newcommand{\Fp}{{\mathbb{F}_p}}

\newcommand{\Char}{\operatorname{char}}
\newcommand{\tors}{\operatorname{tors}}

\newcommand{\Tr}{\operatorname{Tr}}

\newcommand{\Gal}{\operatorname{Gal}}
\newcommand{\GL}{\operatorname{GL}}

\newcommand{\isom}{ \cong }

\newcommand{\OK}{{\mathcal{O}_K}}

\newcommand{\Q}{{\mathbb Q}}

\newcommand{\Z}{{\mathbb Z}}

\newenvironment{Proof}{\par\noindent{\sc Proof:}}%
                      {\hspace*{\fill}\nobreak$\Box$\par\medskip}
                       {\hspace*{\fill}\nobreak$\Box$\par\medskip}

\newtheorem{Proposition}{Proposition}[section]
\newtheorem{Theorem}[Proposition]{Theorem}

\newtheorem{Corollary}[Proposition]{Corollary}

\theoremstyle{definition}

\newtheorem{Remark}[Proposition]{Remark}
 \newtheorem{Example}[Proposition]{Example}

\renewcommand{\baselinestretch}{1.1}

\begin{document}

\title[Divisibility of orders of reductions of elliptic curves]%
{Divisibility of orders of reductions of elliptic curves}

\author[A. Pajaziti]%
{Antigona~Pajaziti}
\address{Department of Mathematics, Faculty of Science, Technology and Medicine,
University of Luxembourg, Luxembourg}\address{Mathematical Institute, Leiden University, Niels Bohrweg 1, 2333 CA Leiden, The Netherlands}
\email{antigona.pajaziti@uni.lu\\
a.pajaziti@math.leidenuniv.nl}
\author[M. Sadek]%
{Mohammad~Sadek}
\address{Faculty of Engineering and Natural Sciences, Sabanc{\i} University, Tuzla, \.{I}stanbul, 34956 Turkey}
\email{mohammad.sadek@sabanciuniv.edu}

\begin{abstract}
Let $E$ be an elliptic curve defined over $\Q$ and $\widetilde{E}_p$ denote the reduction of $E$ modulo a prime $p$ of good reduction for $E$. The divisibility of $|\widetilde{E}_{p}(\Fp)|$ by an integer $m\ge 2$ for a set of primes $p$ of density $1$ is determined by the torsion subgroups of elliptic curves that are $\Q$-isogenous to $E$.   In this work, we give explicit families of elliptic curves $E$ over $\Q$ together with integers $m_E$ such that the congruence class of $|\widetilde{E}_p(\Fp)|$ modulo $m_E$ can be computed explicitly. In addition, we can estimate the density of primes $p$ for which each congruence class occurs. These include elliptic curves over $\Q$ whose torsion grows over a quadratic field $K$ where $m_E$ is determined by the $K$-torsion subgroups in the $\Q$-isogeny class of $E$. We also exhibit elliptic curves over $\Q(t)$ for which the orders of the reductions of every smooth fiber modulo primes of positive density strictly less than $1$ are divisible by given small integers.
\end{abstract}

\maketitle

\let\thefootnote\relax

\section{Introduction}
For an elliptic curve $E$ over a number field $K$, the Mordell-Weil Theorem states that the set of $K$-rational points $E(K)$ of $E$ is a finitely generated abelian group. It was long believed that the torsion subgroup of $E(K)$, $E(K)_{\tors}$, is not only finite but rather it is uniformly bounded by an integer that merely depends on the degree of the number field $K$. Mazur, \cite{Mazur1}, proved that this is indeed the case when $K$ is the rational field $\Q$ giving an exhaustive list of $15$ possibilities for $E(\Q)_{\tors}$. Kenku, Kamienny and Momose, \cite{Kam, Kenku}, established such an exhaustive list for $E(K)_{\tors}$ when $K$ is a quadratic field. Later on, the uniform boundedness of $E(K)_{\tors}$ was proved by Merel, \cite{Merel}, for any number field $K$ of fixed degree. Although the result is exciting, the uniform bound introduced is far from being sharp leaving the mathematical community with the challenge of finding a list of all possible finite abelian groups that may occur as torsion subgroups of elliptic curves over a number field of a given fixed degree. Recently, such a list was presented for cubic fields, \cite{Sporadic}.

Writing $\p$ for a prime ideal in the ring of integers of $K$ with a corresponding residue field $k_{\p}$ of characteristic $p>2$, we set $\widetilde{E}_\p$ to be the reduction of $E$ modulo $\p$. It can be seen that $E(K)_{\tors}$ embeds in $\widetilde{E}_{\p}(k_{\p})$ for every prime $\p$ of good reduction of $E$ satisfying mild ramification conditions. This implies the divisibility of  $|\widetilde{E}_{\p}(k_{\p})|$ by $|E(K)_{\tors}|$ for every such $\p$. One may pose a question on the existence of an elliptic curve over $K$ together with an integer $m\ge 2$ such that $|\widetilde{E}_{\p}(k_{\p})|$ is divisible by $m$ for all but finitely many primes $\p$, or more generally, for a set of primes $\p$ of (natural) density $1$. Serre, \cite{Serre1}, answers this question revealing that such elliptic curves are the ones that are $K$-isogenous to elliptic curves with nontrivial $K$-rational torsion subgroups where $m$ is the order of the torsion subgroup.  In addition, Serre shows that if $|\widetilde{E}_{\p}(k_{\p})|\equiv a\mod m$ for a set of primes $\p$ of density $1$, then $a$ must be $0$ and $m$ must be the order of the torsion subgroup of a $K$-isogenous elliptic curve.   
Consequently, Katz noticed that one may compute the greatest common divisor, $\gcd_{\p\in S}|\widetilde{E}_{\p}(k_{\p})|$, of $|\widetilde{E}_{\p}(k_{\p})|$ where $\p$ runs over a set $S$ of primes of good reduction of $E$ of density one under mild conditions on the ramification indices of the primes, \cite{Katz}. 

The above discussion gives rise to the following questions. If $m\ge 2$ is not a divisor of the order of the torsion subgroup of a $K$-isogenous elliptic curve to $E$, is it possible to determine all the possible congruence classes of $|\widetilde{E}_{\p}(k_{\p})|\mod m$ for all primes $\p$ of good reduction of $E$? In addition, can one compute the density of primes $\p$ at which each congruence class modulo $m$ occurs? 

Fixing the number field $K$ to be the rational field $\Q$ and writing $\Fp$ for the finite field with $p$ elements where $p$ is a rational prime, the number of $\Fp$-rational points $|\widetilde{E}_p(\Fp)|$, where $p$ is a prime of good reduction of $E$, satisfies the Hasse-Weil bound, namely, $\left||\widetilde{E}_p(\Fp)|-p-1\right|\le 2\sqrt{p}$. Deuring, \cite{Deuring}, proved if $E$ has complex multiplication, then $|\widetilde{E}_p(\Fp)|=p+1$ for a set of primes of density $1/2$. In other words, he proved that the set of supersingular primes of an elliptic curve with complex multiplication has density $1/2$. Counting rational points of reductions of Elliptic curves with complex multiplication has received much attention, see for example \cite{Olson,Raj,Padma, Demirci}. Due to the arithmetic properties of complex multiplication, one may find an explicit description of supersingular primes using congruence classes modulo a fixed integer. For example, given a nonzero rational number $D$, the elliptic curve $E^D: y^2 =x^3 -6Dx^2-3D^2x$ has complex multiplication by $\sqrt{-3}$, more precisely $\operatorname{End}(E^D)\isom \Z[\sqrt{-3}]$. The supersingular primes of the curve $E^D$ are exactly the primes $p\equiv 2\mod 3$; whereas if $p\equiv 1 \mod 3$, then $|\widetilde{E}^D_p(\Fp)|=p+1-2c\left(\frac{D}{p}\right)$ where $p=c^2+3d^2$ with $c\equiv\left(\frac{-1}{p}\right)\mod 3$ and $\left(\frac{\cdot}{p}\right)$ is the Legendre symbol modulo $p$, \cite{Poulakis}. In particular, $|\widetilde{E}^D_p(\Fp)|\equiv 0\mod 3$ if $p\equiv 2\mod 3$, while $|\widetilde{E}^D_p(\Fp)|\equiv 0\text{ or }1\mod 3$ if $p\equiv 1\mod 3$. In fact, taking into account the existence of a $2$-torsion point in $E^D(\Q)$, one sees that $|\widetilde{E}^D_p(\Fp)|\equiv 0\mod 6$ if $p\equiv 5 \mod 6$ and $|\widetilde{E}^D_p(\Fp)|\equiv 0\text{ or }4\mod 6$ if $p\equiv 1\mod 6$. It follows that one may describe the congruence classes of $|\widetilde{E}^D_p(\Fp)|\mod 6$ and the congruence classes modulo $6$ of the primes $p$ at which they occur, hence one may get estimates for the densities of these primes $p$ using Dirichlet's Theorem on primes in arithmetic progression. It is worth noting that one of the main tools to obtain the aforementioned expressions for the number of rational points on reductions of elliptic curves $E:\ y^2=f(x)$ with complex multiplication is computing the character sum $\sum_{x\mod p}\left(\frac{f(x)}{p}\right)$. The reason is that the trace of Frobenius $a_p(E):=p+1-|\widetilde{E}_p(\Fp)|$ of the curve $E$ at $p$ can be expressed in terms of the latter character sum. 

Given an integer $m\ge 2$ that is not a divisor of the order of the torsion subgroup of any elliptic curve in the $\Q$-isogeny class of $E$, one knows that $E$ has a torsion point of order $m$ under base change to the division field $\Q(E[m])$, where $E[m]$ is the group of $m$-torsion points on $E$. By considering the splitting behavior of the primes in the latter field, one can easily show that the primes $p$ for which $|\widetilde{E}_p(\Fp)|$ is divisible by $m$ is of positive density. In this work, we consider the case when $|E(K)_{\tors}|\equiv 0\mod m$ for some quadratic field $K$. In \cite{Enrique,Enrique2, Najman}, the possible torsion subgroups of $E$ that appear when the base is changed from $\Q$ to a quadratic field $K$ are listed. Using the arithmetic of quadratic fields, it can be seen that orders of reductions of $E$ modulo primes of good reduction are closely linked to the torsion subgroups of certain quadratic twists of $E$. This enables us to obtain information about the congruence classes of $|\widetilde{E}_p(\Fp)|$ modulo $|E(K)_{\tors}|$. In addition, we show the existence of an integer $N$ that depends on $K$ such that for primes $p$ modulo $ N$, one can determine explicitly the possibilities for $|\widetilde{E}_p(\Fp)|$ modulo $|E(K)_{\tors}|$. Thus, given any $a$ modulo $|E(K)_{\tors}|$, an estimate for the density of primes $p$ such that $|\widetilde{E}_p(\Fp)|\equiv a \mod |E(K)_{\tors}|$ can be computed. The following result will be proved in the paper. 

\begin{Theorem} 
Let $K=\Q(\sqrt{d})$, where $d$ is a square free integer. Let $E$ be an elliptic curve defined over $\mathbb Q$. Set $M_K(E)=\sup\{|E'(K)_{\tors}|: E'\textrm{ is $\Q$-isogenous to }E\}$. Assume that $M_K(E)>M_\Q(E)$. 
If $p\nmid 2d M_K(E)$ is a prime of good reduction of $E$, then
$$ |\widetilde{E}_p(\mathbb F_p)|\equiv \begin{cases}
0 \mod M_K(E) &\text{ if }\left(\frac{d}{p}\right)=1\\
2p+2 \mod M_\Q(E^d) &\text{ if } \left(\frac{d}{p}\right)=-1,\\
\end{cases}$$
where $E^d$ is the quadratic twist of $E$ by $d$.
\end{Theorem}

The latter result provides families of elliptic curves $E$, with no complex multiplication, for which one can find integers $m$ such that $|\widetilde{E}_p(\Fp)|\mod m$ are completely determined and the primes of occurrence $p$ are explicitly known. Up to the knowledge of the authors, there are only few examples in the literature of pairs of elliptic curves with no complex multiplication together with integers $m$ for which the latter information is available. For such an example see \cite[Theorem 2]{KKP} where the congruence classes $|\widetilde{E}_p(\Fp)|\mod 12$ of the curve $E:\ y^2=x^3-12x-11$ are given. We can recover the occurrences of these classes by noticing the growth of the torsion subgroup of a $\Q$-isogenous elliptic curve over the quadratic field $\Q(\sqrt{5})$.

For an elliptic curve over $\Q$ with a nontrivial torsion subgroup of order divisible by $N\ge 1$, the trace of Frobenius of $E$ at a prime $p$ of good reduction satisfies that $a_p(E)\equiv p+1\mod N$. Therefore, unless $N=p$, one sees that the trace of Frobenius $a_p(E)$ can never be $1$ due to a congruence obstruction modulo $ N$. In fact, the Weak Lang-Trotter conjecture, see \cite{Lang} and \cite[Conjecture 1.3]{Katz2}, asserts that it is only congruence obstructions
that prevent an integer $a$ from being a trace of Frobenius $a_p(E)$ of $E$ for infinitely many primes $p$ of good reduction of $E$. In \cite{David}, it was shown that if $E$ has a rational $\ell$-isogeny, $\ell\ne 11$, the
number of primes $p$ such that $a_p(E) \equiv r \mod \ell$ is finite, for some $r$ modulo $\ell$, if and only if $E$ has rational
$\ell$-torsion over the cyclotomic field $\Q(\zeta_{\ell})$.
In our work, for elliptic curves $E$ whose torsion subgroups grow over a quadratic field $K$, we provide other integers $N$, namely divisors of $|E(K)_{\tors}|$, such that congruences modulo $N$ obstruct certain integers from being the trace of Frobenius $a_p(E)$ for infinitely many primes $p$ of good reduction.     

Kubert, \cite{Kubert}, parametrized elliptic curves defined over the rational field $\Q$ possessing nontrivial torsion subgroups. In other words, it was proved that any such elliptic curve is a rational specialization of an elliptic surface defined over $\Q$. In \cite{Rabarison}, elliptic curves with nontrivial torsion subgroups over quadratic fields were parametrized in a similar fashion. We present two explicit examples of elliptic curves over $\Q(t)$ with trivial rational torsion in which a positive proportion of their smooth fibers satisfy the following property. The orders of the reductions of any of these fibers are divisible by a certain integer modulo primes lying in a set of primes of positive density strictly less than $1$.  

\subsection*{Acknowledgments} We would like to thank G\"okhan Soydan, Mohamed Wafik and Tu\v{g}ba Yesin for several comments and suggestions. This work was initiated when A. Pajaziti was a master's student at Sabanc{\i} University under the supervision of M. Sadek. M. Sadek is supported by The Scientific and Technological Research Council of Turkey, T\"{U}B\.{I}TAK; research grant: ARDEB 1001/122F312 and BAGEP Award of the Science Academy, Turkey.

\section{Density $1$ known results} 
\label{sec:density 1}

Throughout this work, $K$ will denote a number field with ring of integers $R$. If $\mathfrak{p}$ is a prime ideal in $R$, then we write $k_\mathfrak{p}$ for the residue field $R/\mathfrak{p}$, and we set $p=\Char k_{\p}$. We denote the corresponding discrete $\p$-valuation by $\nu_{\p}$.  

Given an elliptic curve $E$ defined over $K$, we write $S_E$ for the set of primes of bad reduction of $E$, and $\widetilde{E}_{\p}$ for the reduction of $E$ mod $\p$.

Let $m\geq 2$. The following result classifies all elliptic curves $E$ over $K$ such that 
$|\widetilde{E}_{\p}(k_{\p})|\equiv 0 \mod{m}$ for primes $\mathfrak{p}$ in $K$ of density $1$, see \cite[IV-6]{Serre1} or \cite[Theorem 2]{Katz}.
\begin{Theorem} \label{thm:density1}
 Let $m\ge 2$ be an integer. Let $E$ be an elliptic curve defined over $K$.  The following statements are equivalent: 
 \begin{itemize}
 \item[a)] $|\widetilde{E}_{\p}(k_{\p})| \equiv 0 \mod{m}$ for a set of primes $\p$ of density $1$ in $K$.
 \item[b)] There exists an elliptic curve $E'$ over $K$ such that:
 \begin{itemize}
 \item[i)] $E$ is $K$-isogenous to $E'$; and
 \item[ii)] $|E'(K)_{\tors}|\equiv 0\mod{m}$. 
 \end{itemize} 
 \end{itemize}
\end{Theorem}
If $\beta$ is an integer, $0\le \beta\le m-1$, such that $|\widetilde{E}_{\p}(k_{\p}) |\equiv \beta \mod{m}$ for primes $\p$ of density $1$ in $K$, then
 this is equivalent to $1+\det(\sigma)-\Tr(\sigma)\equiv \beta \mod{m}$ for all $\sigma \in G(E,m)$, where $G(E, m)\subset \GL(2,\mathbb{Z}/m\mathbb{Z})$ is the subgroup defined by the action of the absolute Galois group $\Gal(\overline{K}/K)$ on the set of $m$-torsion points $E[m]$ of $E$, \cite[Theorem 2.3]{Illusie}. Taking $\sigma$ to be the identity matrix yields that $\beta\equiv 0\mod m$.
 Setting 
{\footnotesize\[\mathcal{S}_d=\{m: \text{there is an elliptic curve $E$ defined over a number field } K \text{ with }
[K:\mathbb{Q}]=d \text{ and }  |E(K)_{\tors}|\equiv 0\mod m\},\]
}
one sees that 
\begin{eqnarray*}
    \mathcal{S}_1&=&\{2,3,4,5,6,7,8,9,10,12,14,16\},\\
    \mathcal{S}_2&=&\{2,3,4,5,6,7,8,9,10,11,12,13,14,15,16,18,20,24\},\\
    \mathcal{S}_3&=&\{2,3,4,5,6,7,8,9,10,11,12,13,15,18,20,21,24,28\},
\end{eqnarray*}
see \cite{Mazur1} for $d=1$, \cite{Kenku} for $d=2$, and \cite{Sporadic} for $d=3$. In Theorem \ref{thm:density1}, $m\in S_d$ when $[K:\Q]=d$, $d=1,2,3$.
Fixing an integer $m\ge 2$, it follows that if $E$ is not $K$-isogenous to an elliptic curve $E'$ with $|E'(K)_{\tors}|\equiv 0\mod m$, then there are at least two possible values $\beta_i \mod{m}$, $i=1,2$, such that $|\widetilde{E}_{\p}(k_{\p})|\equiv \beta _i \mod{m}$ for all $\mathfrak{p}$ in a set $S_i$, $i=1,2$, of primes of positive density.
In particular, if $E$ is not $K$-isogenous to an elliptic curve whose torsion subgroup has even order, then $|\widetilde{E}_{\p}(k_{\p})|$ is even, respectively odd, for a set of primes of positive density $\delta_1$, respectively $\delta_2$; and $\delta_1+\delta_2=1$.

 Let $S$ be a set of primes in $K$. Define 
$\gcd_{\p\in S} |\widetilde{E}_{\p}(k_{\p})|$ to be the greatest common divisor
of $|\widetilde{E}_{\p}(k_{\p})|$ where $\p$ runs over $S$.  
One has the following result.
\begin{Theorem}\cite[Theorem 2 (bis)]{Katz}
\label{thm:Katz}
Let $E$ be an elliptic curve over a number field $K$. Let $S$ be any set
of primes of $K$ of density $1$ which consists entirely of primes $\p$ at which $E$ has
good reduction and whose ramification indices $e_{\p}$ satisfy $e_{\p} < p-1$ (e.g. $S$
is the set of all odd unramified primes of good reduction for $E$). Then we have
$$\gcd_{\p\in S} |\widetilde{E}_{\p}(k_{\p})|=\sup\{|E'(K)_{\tors}|: E'\textrm{ is $K$-isogenous to }E\}.$$
\end{Theorem}
If $[K:\Q]=d$, $d=1,2,3$, it follows that $\gcd_{\p\in S} |\widetilde{E}_{\p}(k_{\p})|$ is an integer in $\mathcal S_d \cup\{1\}$ when $S$ is a set of density $1$. When $S$ is the set of all rational primes, the reader may consult \cite{PS}.

In this work, we will be concerned with $\gcd_{\p\in S} |\widetilde{E}_{\p}(k_{\p})|$ when the set $S$ has positive density strictly less than $1$.

\section{Local results over quadratic fields}
In this section, we discuss the relation between the size of the reduction of an elliptic curve $E$ modulo a rational prime $p$ of good reduction of $E$ and the size of the reduction of $E$ modulo a prime that lies above $p$.

Let $K$ be a number field of degree $d$ with ring of integers $\OK$. Let $p$ be a rational prime. If the prime ideal factorization of $p$ in $\OK$ is given by $p=\prod \p_i^{e(\p_i|p)}$, we call $e(\p_i|p)$ the ramification index of $\p_i$ over $p$,  and $f(\p_i|p)=[k_{\p_i}:\Fp]$ the inertial degree of $\p_i$ over $p$,  where $\sum_ie(\p_i|p)f(\p_i|p)=d$. If $e(\p_i|p)>1$ for some $i$, then $p$ is said to ramify in $K$, otherwise $p$ is said to be unramified in $K$. It is known that there are only finitely many rational primes that ramify in a number field $K$, \cite[Chapter 3]{Marcus}.  

We will write $E^d$ for the quadratic twist of $E$ by $d$, where $d$ is an element in the base field. 
\begin{Proposition}
\label{unramified} Let $E$ be an elliptic curve defined over $\Q$. Let $\p$ be a prime in $K$ of good reduction of $E$ over $K$. Assume that $\p\mid p$ where $p$ is unramified in $K$. The following hold.
\begin{itemize}
\item[i)] If $f(\p|p)=1$, then $ \widetilde{E}_{\p}(k_{\p})\isom \widetilde{E}_p(\Fp)$. 
\item[ii)] If $f(\p|p)=2$, then $| \widetilde{E}_{\p}(k_{\p})|=|\widetilde{E}_p(\mathbb F_p)| \cdot |\widetilde{E}_p^d(\mathbb F_p)|$, where $d\in\Fp$ is a nonsquare.  
\item[iii)] If $f(\p|p)=3$, then $| \widetilde{E}_{\p}(k_{\p})|=|\widetilde{E}_p(\mathbb F_p)| \left(p^2-p+1+(p+1)|\widetilde{E}_p^d(\mathbb F_p)|-|\widetilde{E}_p(\mathbb F_p)|\cdot |\widetilde{E}_p^d(\mathbb F_p)|\right)$, where $d\in\Fp$ is a nonsquare.  
\end{itemize}
\end{Proposition}
\begin{Proof}
i) Since $p$ is unramified in $K$ and $f(\p|p)=1$, it follows that $|k_{\p}|=p$, hence $k_{\p}\isom \Fp$. 

For ii), if $f(\p|p)=2$, then $k_{\p}\isom \mathbb{F}_{p^2}$. We recall that 
$|\widetilde{E}_{\p}(k_{\p})|= 1+p^2-a_\mathfrak{p}(E)$, where $a_\mathfrak{p}(E)$ is the trace of Frobenius of $E$ at $\mathfrak{p}$ and $N(\mathfrak{p})=p^2$. If $E$ is defined over $\Q$ has good reduction at a rational prime $p$, then one may use the following recurrence 
$$a_{n+2}(E)=a_1(E) a_{n+1}(E)-p a_n(E), \text{ for all } n\geq 0$$
where $a_0=2$ and $a_1=1+q- |\widetilde E(\mathbb F_q)|$ to compute the trace of Frobenius $a_n(E)$ of $E$ over extensions $\mathbb F_q$, $q=p^n$, see \cite[Chapter V, Exercise 5.13]{silverman}. It follows that 
\begin{eqnarray*}
a_\mathfrak{p}(E)= (p+1-|\widetilde{E}_p(\Fp)|)^2-2p=p^2+1-2(p+1)|\widetilde{E}_p(\Fp)|+|\widetilde{E}_p(\Fp)|^2.
\end{eqnarray*}
It follows that 
\begin{eqnarray*}
|\widetilde{E}_{\p}(k_{\p})|&=&|\widetilde{E}_p(\Fp)|\left(2p+2-|\widetilde{E}_p(\Fp)|\right).
\end{eqnarray*}
We recall that for $d\in\Fp$, if $d$ is not a square, then one has $\widetilde{E}_p^d(\Fp)=2p+2-\widetilde{E}_p(\Fp)$, see
\cite[Proposition 3.21]{Schmitt}.
A similar argument yields iii).

In general, if $E_n=|\widetilde{E}_{\p}(k_{\p})|$ when  $f(\p|p)=n$, then substituting in the recurrence above, one obtains $$E_{n+2}=E_{n+1} + p\cdot E_{n+1} - p\cdot E_n  +|\widetilde{E}_p(\Fp)| (1+ p^{n +1} - E_{n+1}  )).$$ 
\end{Proof}
In fact, a simple induction argument using Proposition \ref{unramified} iii) implies the straight forward result $|\widetilde{E}_{\p}(k_{\p})|\equiv 0\mod |\widetilde{E}_p(\Fp)|$.
The followings are direct consequences of Proposition \ref{unramified}.
\begin{Corollary}
\label{cor1}
Let $E$ be an elliptic curve defined over $\Q$. Let $\p$ be a prime in $K$ of good reduction of $E$ over $K$. Assume that $\p\mid p$ where $p$ is unramified in $K$ and $f(\p|p)=2$. If $\nu_2(|\widetilde{E}_p(\mathbb F_p)|)=r\ge 1 $, then $\nu_2(| \widetilde{E}_{\p}(k_{\p})|) \ge r+1$; whereas if $\nu_2(|\widetilde{E}_p(\mathbb F_p)|)=0 $, then $\nu_2(| \widetilde{E}_{\p}(k_{\p})|) =0$.
\end{Corollary}

\begin{Corollary}\label{halfofprimes}
Let $E$ be an elliptic curve defined over $\Q$. Let $p$ be a prime of good reduction  of $E$. Let $K=\mathbb Q(\sqrt{d})$ where $d$ is a square free integer, and $\mathfrak{ p}\mid p$ be a prime in $K$. Then,
$$|\widetilde{E}_{\p}(k_{\p})|=\begin{cases}
|\widetilde{E}_p(\mathbb F_p)| & \text{ if $\left(\frac{d}{p}\right)=1,\,p\ne 2$ } \\
|\widetilde{E}_p(\mathbb F_p)| \cdot| \widetilde{E}_p^d(\mathbb F_p)| & \text{ if $\left(\frac{d}{p}\right)=-1,\, p\ne 2$, } 
\end{cases}$$
where $\left(\frac{d}{p}\right)$ is the Legendre symbol of $d$ modulo $p$.
\end{Corollary}
\begin{Proof}
This follows from Proposition \ref{unramified} by observing that $p\ne 2$ splits in $K$ if $\left(\frac{d}{p}\right)=1$, 
  whereas $p\ne 2$ is inert in $K$ if $\left(\frac{d}{p}\right)=-1$.
\end{Proof}

\section{Congruence classes of orders of reduction}
In this section, we study the link between the growth of the torsion of an elliptic curve $E$ over $\Q$ under a base change over a number field of small degree and the congruence classes of orders of reduction of $E$ modulo rational primes.

 Theorem \ref{thm:density1} asserts that the statement $E$ is 
$\Q$-isogenous to an elliptic curve $E'$ with $|E'(\Q)_{\tors}|\equiv 0\mod m$ is equivalent to $|\widetilde{E}_p(\Fp)|\equiv 0\mod{m}$, $m\ge2$, for all primes $p$ in a set of primes of density $1$. The following result shows that if $E$ is not $\Q$-isogenous to an elliptic curve whose torsion order is divisible by $m$, then the set of primes $p$ for which $|\widetilde{E}_p(\Fp)|\equiv 0\mod{m}$ is still of positive density. We remark that the result is well-known among experts, but we provide the proof for self-containment.
\begin{Proposition}\label{prop:densitydivisibility}
Let $E$ be an elliptic curve over $\Q$ and $m>1$ be an integer such that 
$E$ is not $\Q$-isogenous to an elliptic curve $E'$ with $|E'(\Q)_{\tors}|\equiv 0\mod m.$ Then $|\widetilde{E}_p(\Fp)|\equiv 0\mod{m}$ for all primes $p$ in a set $S$ of primes of positive density $\delta<1$.
\end{Proposition}
\begin{Proof}
One considers the Galois number field $F=\Q(E[m])$ obtained by adjoining the $x$- and $y$- coordinates of the $m$-torsion points of $E$ to $\Q$, where the $x$-coordinates are the roots of the $m$-th division polynomial $\psi_m$ of $E$. According to Chebotarev Density Theorem, the primes that split completely in $F$ are of density at least $1/|\Gal(F/\Q)|$. Since $|E(F)_{\tors}|\equiv 0\mod m$, Proposition \ref{unramified} i) implies that $|\widetilde{E}_p(\Fp)|\equiv 0\mod{m}$ for a set of primes of density at least $1/|\Gal(F/\Q)|$.
\end{Proof}
We remark that the statement in Proposition \ref{prop:densitydivisibility} is valid if $\Q$ is replaced by any number field $K$.

In this section, we will establish the existence of families of elliptic curves $E$ over $\Q$ together with integers $m$ such that one can explicitly compute the density of primes $p$ for which $|\widetilde{E}_p(\Fp)| \equiv r\mod m$, $0\le r\le m-1$. For this purpose, we focus on elliptic curves whose torsion subgroups grow over quadratic field extensions. 


In a series of papers \cite{Enrique,Enrique2},  the authors answer the following question. Given an elliptic curve over $\Q$ with torsion group $E(\Q)_{\tors}$, considering a base change of $E$ to a quadratic number field $K$, what are the possibilities for $E(K)_{\tors}$? Moreover, they listed the possible number of quadratic fields over which the torsion of $E$ grows together with the groups that are realized as torsion groups of $E$ over these fields. In the following subsection, we write down the results that will be used in this work. 

Let $\Phi$ be the set of possible groups that can appear as the torsion subgroup of an elliptic
curve defined over $\Q$. If $G\in\Phi$, we write $\Phi(d, G)$, $d\ge 2$, for the set of possible groups that can
appear as the torsion subgroup over a number field $K$ of degree $d$ of an elliptic curve
$E$ defined over $\Q$ such that $E(\Q)_{\tors}=G$. The following is \cite[Theorem 1]{Enrique2}.
\begin{Proposition}
\label{prop:En}
Let $G\in\Phi$. The set $\Phi(2,G)$ is described as follows 

\begin{table}[H]
        \begin{tabular}{|l| l|}
\hline
$G$& $\Phi(2,G)$\\
\hline
$\{0\}$ & $\{0\}$, $\Z/3\Z$, $\Z/5\Z$, $\Z/7\Z$, $\Z/9\Z$\\
 \hline 
 $\Z/2\Z$ & $\Z/2\Z$, $\Z/4\Z$, $\Z/6\Z$, $\Z/8\Z$, $\Z/10\Z$, $\Z/12\Z$, $\Z/16\Z$, $\Z/2\Z\times \Z/2\Z$, $\Z/2\Z\times \Z/6\Z$, $\Z/2\Z\times \Z/10\Z$\\
        \hline 
        $\Z/3\Z$ &$\Z/3\Z$, $\Z/15\Z$, $\Z/3\Z\times\Z/3\Z$\\
        \hline
        $\Z/4\Z$ & $\Z/4\Z$, $\Z/8\Z$, $\Z/12\Z$, $\Z/2\Z\times \Z/4\Z$, $\Z/2\Z\times \Z/8\Z$, $\Z/2\Z\times \Z/12\Z$, 
     $\Z/4\Z\times \Z/4\Z$\\
    \hline 
$\Z/5\Z$ & $\Z/5\Z$, $\Z/15\Z$ \\
\hline
$\Z/6\Z$ &$\Z/6\Z$, $\Z/12\Z$, $\Z/2\Z\times \Z/6\Z$, $\Z/3\Z\times \Z/6\Z$\\
 \hline 
 $\Z/7\Z$ & $\Z/7\Z$\\
\hline 
 $\Z/8\Z$ &$\Z/8\Z$, $\Z/16\Z$, $\Z/2\Z\times \Z/8\Z$\\
 \hline
 $\Z/9\Z$ & $\Z/9\Z$\\
 \hline
 $\Z/10\Z$ &$\Z/10\Z$, $\Z/2\Z\times \Z/10\Z$\\
\hline
 $\Z/12\Z$ &$\Z/12\Z$, $\Z/2\Z\times \Z/12\Z$\\
\hline
        $\Z/2\Z\times\Z/2\Z$ &$\Z/2\Z\times\Z/2\Z$, $\Z/2\Z\times \Z/4\Z$, $\Z/2\Z\times \Z/6\Z$, $\Z/2\Z\times \Z/8\Z$, $\Z/2\Z\times \Z/12\Z$\\
         \hline
         $\Z/2\Z\times\Z/4\Z$ &$\Z/2\Z\times\Z/4\Z$, $\Z/2\Z\times \Z/8\Z$, $\Z/4\Z\times \Z/4\Z$\\
    \hline
     $\Z/2\Z\times\Z/6\Z$ &$\Z/2\Z\times\Z/6\Z$, $\Z/2\Z\times \Z/12\Z$\\
     \hline
     $\Z/2\Z\times\Z/8\Z$ &$\Z/2\Z\times\Z/8\Z$\\
     \hline
\end{tabular}\\
\end{table}
\end{Proposition}

A special case of \cite[Lemma 1.1]{Laska} asserts that if $K=\Q(\sqrt{d})$, where $d$ is a square free integer, then there exist homomorphisms 
\begin{eqnarray}\label{eq1} 
E(\Q)\oplus E^d(\Q)\to E(K),\qquad E(K)\to E(\Q)\oplus E^d(\Q),
\end{eqnarray}
where the kernels and cokernels of both homomorphisms are annihilated by the multiplication by-$2$-map, see \cite[Lemma 1]{Najman}. In particular, if $n$ is an odd integer, then 
\[E(\Q)[n]\oplus E^d(\Q)[n]\isom E(K)[n].\]
We are now in a place to discuss families of elliptic curves for which we can find a positive integer $m\ge 2$ such that the congruence classes of the orders of its reductions are explicitly determined modulo $m$, hence we can compute the density at which each class occurs.  

Given an elliptic curve $E$ defined over a number field $K$ then we define $$M_K(E)=\sup\{|E'(K)_{\tors}|: E'\textrm{ is $\Q$-isogenous to }E\}.$$

\begin{Theorem} \label{Oddorder}
Let $K=\Q(\sqrt{d})$, where $d$ is a square free integer. Let $E$ be an elliptic curve defined over $\mathbb Q$ such that $M_K(E)>M_\Q(E)$.  Assume moreover that $\ell$ is an odd integer such that $ \ell\mid M_K(E)$ but $\ell\nmid M_{\Q}(E)$. If $p\nmid 2d M_K(E)$ is a prime of good reduction of $E$, then
$$ |\widetilde{E}_p(\mathbb F_p)|\equiv \begin{cases}
0 \mod \ell &\text{ if }\left(\frac{d}{p}\right)=1\\
2p+2 \mod \ell &\text{ if } \left(\frac{d}{p}\right)=-1.\\
\end{cases}$$
In particular, the density of primes $p$ such
that $| \widetilde{E}_p(\Fp)| \equiv 0 \mod \ell$ is at least $1/2$ .
\end{Theorem}
\begin{proof}
If the prime $p$ splits in $K$, i.e., $\left(\frac{d}{p}\right)=1$, then it follows by Corollary \ref{halfofprimes} that $|\widetilde{E}_p(\mathbb F_p)|\equiv |\widetilde{E}_{\p}(k_{\p})| \mod \ell$ for every prime $\p$ lying above $p$.  Since $|E'(K)_{\tors}|\equiv 0\mod \ell$ for some elliptic curve $E'$ that is $\Q$-isogenous to $E$, it follows that $|\widetilde{E}_{\p}(k_{\p})|\equiv 0\mod \ell$ for every prime $\p$ of good reduction of $E$ over $K$ that lies above a rational prime that splits in $K$, see Theorem \ref{thm:density1}.

Now we assume that the prime $p$ is inert in $K$, i.e., $\left(\frac{d}{p}\right)=-1$. One knows that $|\widetilde{E}_p(\Fp)|=2p+2-|\widetilde{E}_p^d(\Fp)|$.  Since $\ell$ is odd, the paragraph before the theorem implies that $|E'^d(\Q)_{\tors}|\equiv 0\mod \ell$. For these primes $p$, Theorem \ref{thm:density1} implies that $|\widetilde{E}_p^d(\Fp)|\equiv 0\mod\ell$.  
\end{proof}

We recall that for a prime $p\ge 5$, if the trace of Frobenius $a_p(E)=p+1-|\widetilde{E}_p(\Fp)|$ of $E$ at a prime of good reduction $p$ is such that $a_p(E)=0$, then $p$ is said to be a {\em supersingular} prime for $E$. Theorem \ref{Oddorder} yields the following result of possible congruence values of trace of Frobenius of elliptic curves satisfying the assumptions of Theorem \ref{Oddorder}.
\begin{Corollary} \label{cor:supersingular}
Let $E$, $K$, $d$ and $\ell$ satisfy the hypotheses of Theorem \ref{Oddorder}. Let $E'$ be an elliptic curve $\Q$-isogenous to $E$. Let $p\nmid 2d|E'(K)_{\tors}|$ be a prime of good reduction of $E$. If $p$ is a supersingular prime for $E$, then $p\equiv -1 \mod \ell$.
\end{Corollary}

\begin{Example}
Let $E$ be an elliptic curve defined by $y^2+xy=x^3+x^2-700x+34000$ over $\mathbb Q$. One may check that $E(\mathbb{Q})_{\tors}\cong \mathbb Z/2 \mathbb Z$ and $E(\mathbb Q(\sqrt{5}))_{\tors}\cong \mathbb Z/10 \mathbb Z$. It follows that the torsion subgroup of the quadratic twist $E^5$ is $E^5(\mathbb{Q})_{\tors} \cong \mathbb Z/10 \mathbb Z$. Thus, by Theorem \ref{Oddorder}, one has that
$$|\widetilde{E}_p(\mathbb F_p)|\equiv \begin{cases}
0 \mod{10} &\text{ if } p\equiv 1,4 \mod 5 \\
6 \mod{10} &\text{ if } p\equiv 2 \mod 5 \\
8 \mod{10}&\text{ if } p\equiv 3 \mod 5.
\end{cases}$$

We notice that $|\widetilde{E}_p(\mathbb F_p)|\equiv 0 \mod 10$ for primes $p$ of density $\frac{1}{2}$, $|\widetilde{E}_p(\mathbb F_p)|\equiv 6 \mod{10}$ for primes of density $\frac{1}{4}$, and $|\widetilde{E}_p(\mathbb F_p)|\equiv8 \mod{ 10}$  for primes of density $\frac{1}{4}$.

In addition, if $p$ is a supersingular prime for $E$, then $p\equiv 9\mod{10}$. More precisely, one has the following formula for the trace of Frobenius $a_p(E)=p+1-|\widetilde{E}_p(\Fp)|$, where $p$ is a prime of good reduction of $E$
$$a_p(E)\equiv \begin{cases}
0 \mod{10} &\text{ if } p\equiv 9 \mod 10 \\
2\mod{10} &\text{ if } p\equiv 1, 7 \mod 10 \\
6 \mod{10}&\text{ if } p\equiv 3 \mod 10.
\end{cases}$$
\end{Example}

More generally, one has the following generalized version of Theorem \ref{Oddorder}. The proof is similar to the proof of Theorem \ref{Oddorder}.

\begin{Theorem} \label{Order}
Let $K=\Q(\sqrt{d})$, where $d$ is a square free integer. Let $E$ be an elliptic curve defined over $\mathbb Q$ such that $M_K(E)>M_\Q(E)$. 
If $p\nmid 2d M_K(E)$ is a prime of good reduction of $E$, then
$$ |\widetilde{E}_p(\mathbb F_p)|\equiv \begin{cases}
0 \mod M_K(E) &\text{ if }\left(\frac{d}{p}\right)=1\\
2p+2 \mod M_\Q(E^d) &\text{ if } \left(\frac{d}{p}\right)=-1.\\
\end{cases}$$
In particular, the density of primes $p$ such
that $| \widetilde{E}_p(\Fp)| \equiv 0 \mod M_K(E)$ is at least $1/2$ .
\end{Theorem}




The following example, conjectured in \cite{Sun} and proved in \cite[Theorem 2]{KKP}, displays an elliptic curve over $\Q$ where the congruence classes of orders of reductions of the curve are determined completely modulo the integer $12$ according to the congruence classes of the primes of good reduction when considered modulo $20$. We give an explanation for the occurrence of these congruence classes using the results above. 
\begin{Example}
Let $E: y^2=x^3-12x-11$ be an elliptic curve defined over $\mathbb{Q}$. The elliptic curve $E$ is $\mathbb{Q}$-isogenous to the elliptic curve $E':y^2=x^3 - 372x + 2761$ with torsion $\mathbb{Z}/6\mathbb{Z}$. Therefore, one knows that $|\widetilde{E}_p(\mathbb{F}_p)|\equiv0\mod 6 $ for all primes of good reduction, see Theorem \ref {thm:density1}. If one considers $|\widetilde{E}_p(\mathbb{F}_p)|\mod 12$, then the two possible congruence classes are $0$ and $6$. Moreover, one knows that the congruence $6\mod 12$ must occur with positive density.

Now one considers the curves $E$ and $E'$ over the quadratic extension $K=\Q(\sqrt{5})$. One sees that $E'(K)_{\tors}\isom\Z/2\Z\times\Z/6\Z$. 
According to Theorem \ref{Order}, one has that $|\widetilde{E}_p(\mathbb{F}_p)|\equiv 0\mod 12$ for primes $p\equiv 1,4\mod 5$. For the primes $p\equiv 2,3 \mod 5$, one may have either possibilities for $|\widetilde{E}_p(\mathbb{F}_p)|\mod 12$. For this reason we investigate the divisibility of $|\widetilde{E}_p(\mathbb{F}_p)|$ by $4$ in each case. For $|\widetilde{E}_p(\mathbb{F}_p)|$ to have a point of order $4$, the $2$-torsion point $(-1,0)$ must be divisible by $2$, i.e., there is a point $P\in \widetilde{E}_p(\mathbb{F}_p)$ such that $2P=(-1,0)$. Using the duplication formula on $E$, one knows that the $x$-coordinate of $P$ must be a root of the polynomial $(x^2+2x+10)^2$. The latter is equivalent to $-1$ being a square modulo $p$, i.e., $p\equiv 1\mod 4$.
Therefore, one has 
\begin{equation*}|\widetilde{E}_p(\mathbb{F}_p)|\equiv\begin{cases} 0 \mod 12 &\text{ if } p\equiv 1,9,11,13,17,19 \mod 20\\
6 \mod 12& \text{ if } p\equiv 3,7 \mod 20.
\end{cases}
\end{equation*}
In particular, $|\widetilde{E}_p(\mathbb{F}_p)|\equiv 0\mod 12$ for primes of density $3/4$, whereas 
$|\widetilde{E}_p(\mathbb{F}_p)|\equiv 6\mod 12$ for primes of density $1/4$.
\end{Example}

\begin{Example}
Let $E$ be the elliptic curve defined by $y^2+xy+y=x^3-x^2+47245x-2990253$  over $\mathbb Q$. It is easily seen that $E(\mathbb{Q})_{\tors}\cong \mathbb Z/2 \mathbb Z$ and $E(\mathbb Q(\sqrt{-15}))_{\tors}\cong \mathbb Z/16 \mathbb Z$. Also, the torsion subgroup of the quadratic twist $E^{-15}$ is given by $E^{-15}(\mathbb{Q})_{\tors} \cong \mathbb Z/8 \mathbb Z$. Moreover, if $p\equiv 7,11,13,14 \mod 15$, or equivalently $\left(\frac{-15}{p}\right)=-1$, then $|\widetilde{E}_p^{-15}(\mathbb F_p)|\equiv 0  \mod 8$. It follows that in the latter case $|\widetilde{E}_p(\Fp)|\equiv0,2,4,6 \mod 8$. However, the elliptic curve $E$ is $\Q$-isogenous to the curve $y^2 + xy + y = x^3 - x^2 - 240755x - 26606253$ whose torsion subgroup is isomorphic to $\Z/2\Z\times\Z/2\Z$. This implies that $|\widetilde{E}_p(\Fp)|\equiv 0\mod 4$ for any prime $p$ of good reduction of $E$. Thus,
Theorem \ref{Order} yields that
\begin{equation*}
|\widetilde{E}_p(\mathbb F_p)|\equiv\begin{cases} 
0 \mod 16 & \text{ if } p\equiv 1,2,4,8 \mod 15 \\
0,4,8,12  \mod 16&\text{ if } p\equiv 7,11,13,14 \mod 15. 
\end{cases}
\end{equation*}
\end{Example}

\begin{Example}
Let $E:y^2=x^3+20148x+586096$ be an elliptic curve defined over $\mathbb Q$ with $E(\mathbb{Q})_{\tors}\cong \mathbb Z/2 \mathbb Z$ and $E(\mathbb Q(\sqrt{-6}))_{\tors}\cong \mathbb Z/8 \mathbb Z$. Also, the torsion subgroup of the quadratic twist of $E^{-6}$ is given by $E^{-6}(\mathbb{Q})_{\tors} \cong \mathbb Z/8 \mathbb Z$. A prime $p$ splits in $\mathbb Q(\sqrt{-6})$ if $p\equiv 1,5,7,11 \mod 24$, otherwise $p$ is inert.
Therefore, Theorem \ref{Order} gives that
$$|\widetilde{E}_p(\mathbb F_p)|\equiv \begin{cases}
0 \mod 8& \text{ if }p\equiv 1,5,7,11,19,23 \mod 24 \\
4 \mod 8& \text{ if } p\equiv 13,17 \mod 24.
\end{cases}$$
One notices that $E$ is $\Q$-isogenous to an elliptic curve whose torsion subgroup is isomorphic to $\Z/2\Z\times\Z/2\Z$.
\end{Example}


The following observations can be found in \cite[Proposition 1]{Najman} or \cite{Enrique2}.
\begin{Remark}
\label{rem1}
\begin{itemize}
\item[i)] If $E(\Q)_{\tors}\isom \Z/3\Z$ is such that $E(K)_{\tors}\isom \Z/3\Z\times\Z/3\Z$, then $K=\Q(\sqrt{-3})$. 
\item[ii)] The curves 50a3 and 450b4 are the only elliptic curves $E$ with $E(\Q)_{\tors}\isom\Z/3\Z$ such that there exists $d=5$ and $d=-15$, respectively, with $E^d(\Q)_{\tors}\isom \Z/5\Z$. Moreover, $E^d(\Q)_{\tors}=\{O_E\}$ for any other square free integer $d$. The reader may refer to the table in Appendix for the explicit computations of the congruence classes of $|\widetilde{E}_p(\Fp)|$ modulo $15$ for the curve $E$ defined over $\Q$ and labeled as 50a3.
\end{itemize}
\end{Remark}

\begin{Corollary}
\label{cor:examples}
Let $E$ be an elliptic curve defined over $\Q$.
\begin{itemize}
\item[i)] If $E(\Q)_{\tors}\isom \Z/3\Z$ is such that $E(K)_{\tors}\isom \Z/3\Z\times\Z/3\Z$, then 
$$|\widetilde{E}_p(\mathbb F_p)|\equiv \begin{cases}
0 \mod 9& \text{ if } p\equiv 1 \mod 3 \\
0,3,6 \mod 9& \text{ if } p\equiv 2\mod 3.
\end{cases}$$
\item[ii)] If $E(\Q)_{\tors}\isom \Z/6\Z$ and $\Z/6\Z\subseteq E^d(\Q)_{\tors}$, then $d=-3$. 
\item[iii)] If $E(\Q)_{\tors}\isom \Z/2\Z\times \Z/4\Z$ and $E^d(\Q)_{\tors}\isom \Z/2\Z\times \Z/4\Z$, then $d=-1$. 
\end{itemize}
\end{Corollary}
\begin{Proof}
i) This is Remark \ref{rem1} i) and Theorem \ref{Order}.

ii) When $\left(\frac{d}{p}\right)=-1$, Theorem \ref{Order} implies that $|\widetilde{E}_p(\Fp)|=2p+2+ s |E^d(\Q)_{\tors}| $ for some $s\in\Z$. Considering the last equality modulo 6, one has $p\equiv 2\mod 3$. It follows that $d=-3$.

For iii) Theorem \ref{Order} implies that $|\widetilde{E}_p(\Fp)|=2p+2+ s |E^d(\Q)_{\tors}| $ for some $s\in\Z$ when $\left(\frac{d}{p}\right)=-1$. Now considering the last equality modulo 8, one has $p\equiv 3\mod 4$. It follows that $d=-1$.
\end{Proof}
One remarks that the possibilities for the elliptic curves in ii) and iii) in Corollary \ref{cor:examples} can be found in \cite[Proposition 2, Proposition 3]{Najman}. It is known that the torsion subgroups $\Z/3\Z\times\Z/3\Z$ and $\Z/4\Z\times\Z/4\Z$ occur only over the quadratic fields $\Q(\sqrt{-3})$ and $\Q(\sqrt{-1})$, respectively, see \cite{Kam, Kenku}. Our method gives the result without relying on these facts.  

\section{Divisibility in families of elliptic curves }

Let $E$ be an elliptic curve defined over a number field $K$. In \cite{Kubert}, Kubert gave a parametrization
of all elliptic curves over $K = \Q$ with a nontrivial torsion subgroup. In particular, it was proved that elliptic curves with a fixed nontrivial torsion subgroup  over $\Q$ constitute an elliptic surface over $\Q$. Similarly, a parametrization of elliptic curves defined over a quadratic field
with a certain torsion subgroup was given in \cite{Rabarison}. 

In this section, we aim at constructing families of elliptic curves $E_t$ over $\Q(t)$ together with a set of primes $S$ of positive density strictly less than $1$ such that for all possible values of $t=t_0\in \Z$, one has $|\widetilde{E}_{t=t_0,p}(\Fp)|$ is divisible by a fixed integer for all $p\in S$. 
As an example, one can find the following family of elliptic curves, \cite[Theorem 1]{KKP}.
 
 \begin{Example}\label{Ex}
 Let $p>3$ be a prime and $t$ be an integer such that $ t(9t+4)\not \equiv 0 \mod p$. Let $E$ be an elliptic curve given by
 $$E_t: y^2 = x^3-(6t + 3)x-(3t^2 +6t +2).$$
 Then $| \widetilde{E}_{t,p}(\mathbb{F}_p)|\equiv 0 \mod{3}$.
 \end{Example}
We recall that if $E_{\textbf{a}}$, $\textbf{a}=(a_1,a_2,a_3,a_4,a_6)$, is an elliptic curve defined by a Weierstrass equation $y^2+a_1xy+a_3y=x^3+a_2x^2+a_4x+a_6$ with discriminant $\Delta_{\textbf{a}}$ , then the division polynomial $\psi_m$, $m\ge2$ is odd, is a polynomial in $\mathbb{Z}[a_1,a_2,a_3,a_4,a_6, x]$ of degree $(m^2-1)/2$, see \cite[Chapter III]{silverman} for the definition and explicit description of $\psi_m$. In addition, if $P=(x(P),y(P))\ne O_E$ is a point in $E$, then $P$ is a point in $E[n]$ if and only if $\psi_n(x(P)) =0 $.

 We will need the following standard proposition for our construction. 
\begin{Proposition}\label{parametrization}
Let $p\ne 2$ be a rational prime. Let $K$ be a number field, with ring of integers $R$, and $\mathfrak{p}$ be a prime of $K$ with $\operatorname{char} k_{\mathfrak{p}}\ne 2,3$.

Fix $T\in R$. Let $\textbf{A}=(A_1,A_2,A_3,A_4,A_6)$ be a $k_\mathfrak{p}$-solution of the polynomial equation $G_T(a_1,a_2,a_3,a_4,a_6)\equiv 0 \mod{ \mathfrak{p}}$, where $$G_T(a_1,a_2,a_3,a_4,a_6)=\psi_{p}(a_1,a_2,a_3,a_4,a_6,T)\in k_\mathfrak{p}[a_1,a_2,a_3,a_4,a_6].$$
If the following two conditions hold
\begin{itemize}
\item[i)] $\Delta_{\textbf{A}}\not\equiv 0 \mod \mathfrak{p}$, and
 \item[ii)] $z_T^2+A_1 T z_T+A_3 z_T\equiv T^3+A_2T^2+A_4 T+A_6 \mod{\mathfrak p}$ for some $z_T \in k_\mathfrak{p}$,
\end{itemize}
then the elliptic curve $E_{\textbf{A}}: y^2+A_1xy+A_3y=x^3+A_2x^2+A_4x+A_6$ over $k_\mathfrak{p}$ satisfies
 $|E_{\textbf{A}}(k_\mathfrak{p})|\equiv 0 \mod{ p}$.
\end{Proposition}
\begin{proof}
Condition i) guarantees that $E_{\textbf{A}}$ is an elliptic curve over $k_\mathfrak p$, whereas condition ii) implies that $P_T=(T\mod \p,z_T)\in E_{\textbf{A}}(k_\mathfrak{p})$. Finally, the fact that $G_T(A_1,A_2,A_3,A_4,A_6)\equiv 0 \mod \mathfrak p$  asserts that $P_T$ is a point of order $p$ in $ E_{\mathbf{A}}(k_\mathfrak{p})$.
\end{proof}

\begin{Theorem}\label{family3}
Consider the elliptic curve  
 \begin{eqnarray*}
 E_t: y^2=f_{t}(x):=x^3 -3( t^2+1)x^2 + 3x-1,\qquad\text{where }\Delta(E_t)=-432\ t^4 (9 + 4 t^2) \ne 0 .
 \end{eqnarray*}
 For any $t\in\Q\setminus\{0\}$, one has
 $|\widetilde{E}_{t,p}(\mathbb{F}_p)|\equiv 0 \mod 3$ for all primes $p\ne 2,3$ such that $\nu_p(t(9 + 4 t^2))=0$.    

 Moreover, there are infinitely many rational values of $t$ such that $|\widetilde{E}_{t,p}(\Fp)|\equiv 0\mod 6$ for a set $S$ of primes $p$ of density at least $2/3$; and $|\widetilde{E}_{t,p}(\Fp)|\equiv 0\mod 12$ for at least half of the primes in $S$. 
\end{Theorem}
\begin{Proof}
Although the elliptic curve $E_t$ does not possess a torsion point of order $3$ over $\Q(t)$, the curve $E_t$ and the curve $y^2 + 6t\,xy + 2t(4t^2+9)y = x^3$
are $3$-isogenous over $\Q(t)$, and the latter has a $3$-torsion point $(0,0)$.

The polynomial $f_{t}$ is irreducible over $\Q(t)$. Moreover, it is easily seen that the discriminant of $f_{t}$ is not a square in $\Q(t)$. It follows that the Galois group of $f_{t}$ over $\Q(t)$ is the symmetric group $S_3$ on a set of three elements. Hilbert's irreducibility theorem implies the existence of infinitely many rational values $t=t_0$ for which $f_{t_0}$ has Galois group $S_3$. In accordance with Chebotarev's density theorem, the density of primes $p$ for which $f_{t_0}$ has at least one root modulo $p$ is $2/3$. It follows that for each such $t=t_0$, the density of primes $p$ such that either $|\widetilde{E}_{t_0,p}(\Fp)|\equiv 0\mod 6$, when $f_{t_0}$ has exactly one root modulo $p$ and hence $\widetilde{E}_{t_0,p}(\Fp)$ has exactly one nontrivial $2$-torsion point; or $|\widetilde{E}_{t_0,p}(\Fp)|\equiv 0\mod 12$, when $f_{t_0}$ splits modulo $p$ and hence $\widetilde{E}_{t_0,p}(\Fp)$ has full $2$-torsion, is of density at least $2/3$. 
\end{Proof}

\begin{Theorem}
\label{family5}
Let $E_t$ be the elliptic curve described by
 \begin{eqnarray*}
 E_t: y^2=g_{t}(x):=x^3 -7t\ x^2+ 96 t^2\ x+256\ t^3,\qquad\text{where }\Delta(E_t)=-121634816\ t^6 \ne 0 .
 \end{eqnarray*}
 For any $t\in\Q\setminus\Q^2$ and any prime $p\ne 2,3,29$ with $\nu_p(t)=0$ such that $\left(\frac{t}{p}\right)=1$, one has
 $|\widetilde{E}_{t,p}(\mathbb{F}_p)|\equiv 0 \mod 5$.

 Moreover, there are infinitely many rational values of $t$ such that $|\widetilde{E}_{t,p}(\Fp)|\equiv 0\mod 10$ for a set $S$ of primes $p$ of density at least $1/6$; and $|\widetilde{E}_{t,p}(\Fp)|\equiv 0\mod 20$ for a positive proportion of the primes in $S$.  
\end{Theorem}
\begin{Proof}
The $5$-th division polynomial $\psi_5$ of the elliptic curve $E_t$ is given by
\begin{eqnarray*}
\psi_5(x)=&-&x\ (16\ t - x)  (343597383680\ t^{10} + 42949672960\ t^9\ x - 
   5368709120\ t^8\ x^2 \\&+& 1929379840\ t^7\ x^3 - 56623104\ t^6\ x^4 + 
   33816576\ t^5\ x^5 - 835584\ t^4\ x^6 + 135936\ t^3\ x^7 \\&+& 
   5776\ t^2\ x^8 - 60\ t\ x^9 + 5\ x^{10}).
   \end{eqnarray*}
Now one obtains $g_t(0)=2^8\ t^3$. Therefore, if $\left(\frac{t}{p}\right)=1$, then $(0,2^4tm_t)\in \widetilde{E}_{t,p}(\Fp)$ where $t\equiv m_t^2\mod p$. We now conclude using Proposition \ref{parametrization}.

   The polynomial $g_{t}$ is irreducible over $\Q(t)$ with Galois group $S_3$.  Hilbert's irreducibility theorem yields the existence of infinitely many rational values $t=t_0$ for which $g_{t_0}$ has Galois group $S_3$. According to Chebotarev's density theorem, the density of primes $p$ for which $g_{t_0}$ has at least one root modulo $p$ is $2/3$. Since the set of primes $p$ for which $\left(\frac{t}{p}\right)=1$ and hence $|\widetilde{E}_{t,p}(\mathbb{F}_p)|\equiv 0 \mod 5$ is $1/2$, one gets that the density of primes $p$ such that either $|\widetilde{E}_{t_0,p}(\Fp)|\equiv 0\mod 10$ or $|\widetilde{E}_{t_0,p}(\Fp)|\equiv 0\mod 20$ is at least $1/6$. 
\end{Proof}
\begin{Remark}
The curve $E_t$ in Theorem \ref{family5} does not possess a rational point of order $5$ over $\Q(t)$. 
One also notices that $E_t$ is the quadratic twist by $t$ of the elliptic curve $E: y^2=x^3 -7\ x^2+ 96\ x+256$. Therefore, for all quadratic twists of $E$ over $\Q$, the order of the reduction of the twist is divisible by $5$ modulo primes of density at least $1/2$.  
\end{Remark}

\section*{Appendix}
\label{sec:computations}
In the following section, we consider elliptic curves $E$ over $\Q$ with torsion subgroup $E(\Q)_{\tors}$ for which there exists a quadratic field $K=\Q(\sqrt{D})$, where $D$ is a square free integer, over which one has $E(\Q)_{\tors}\subsetneq E(K)_{\tors}$. We provide an example of such an elliptic curve $E$ with every possibility of the pair $(E(\Q)_{\tors}, E(K)_{\tors})$ as listed in Proposition \ref{prop:En} together with the congruence classes of $|\widetilde{E}_p(\Fp)|\mod m$, $m=|E(K)_{\tors}|$, and the primes $p$ at which the latter classes occur computed in view of Theorem \ref{Oddorder} and Theorem \ref{Order}. The examples are listed in a table where 
\begin{itemize}
    \item[i)] the first column contains the label of the elliptic curve $E$ with minimal conductor satisfying the conditions above on the pair $(E(\Q)_{\tors}, E(K)_{\tors})$, where the labelling is following Cremona's tables, \cite{Cremona},
    \item[ii)] the second column contains the torsion subgroup $E(\Q)_{\tors}$ of $E$ over $\mathbb Q$,
    \item[iii)] the third column contains $E(K)_{\tors}$,
    \item[iv)] the forth column contains $D$,
    \item[v)] the fifth column contains the congruence classes of $|\widetilde{E}_p(\mathbb F_p)|\mod m$, where $m=|E(K)_{\tors}|$,
    \item[vi)] the sixth column contains the congruence classes of primes $p$ modulo an integer $a$ such that $|\widetilde{E}_p(\mathbb F_p)|\equiv t\mod m$ if $p\equiv s\mod a$. 
\end{itemize}

\begin{center}
\begin{longtable}{|c| c| c|c |l |l |} 
 \hline
 label & $E(\mathbb Q)_{\tors}$ & $E(K)_{\tors}$ & $D$ & $|\widetilde{E}_p(\mathbb{F}_p)|\mod m$ & $p$  \\ 
 \hline\hline
 175b2 & $\{O_E\}$ & $\mathbb Z/ 3 \mathbb Z$ & \multirow{2}{*}{$5$} & {$0 \mod 3$}  & $p\equiv 1,4\mod 5$ \\ \cline{5-6}
   &  &  & & $0,1 \mod 3$ & $p \equiv 2,3 \mod 5$ \\
 \hline
 75a2 & $\{O_E\}$ & $\mathbb Z/ 5 \mathbb Z$ & \multirow{2}{*}{$5$}& $0 \mod 5$ & $p\equiv 1,4\mod 5$\\ \cline{5-6}
  &  &  & & $1 \mod 5$ & $p\equiv 2\mod 5$\\ \cline{5-6}
  &  &  & & $ 3\mod 5$ & $p\equiv 3\mod 5$\\
 \hline
  208d1 & $\{O_E\}$ &$\mathbb Z/ 7\mathbb Z$ & \multirow{2}{*}{$-1$}& $0\mod 7$&  $p\equiv 1 \mod 4$ \\ \cline{5-6}
      &  & & & $0,1,3,4,5,6\mod 7$ & $p\equiv 3 \mod 4$ \\
   \hline
  54a2 & $\{O_E\}$ & $\mathbb Z/ 9 \mathbb Z$ & \multirow{2}{*}{$-3$} & $0\mod 9$ & $p\equiv 1 \mod 3$ \\ \cline{5-6}
  &  & &  & $0,3,6 \mod 9$ & $p\equiv 2 \mod 3$ \\
   \hline
   \hline
 98a4 & $\mathbb Z/ 2 \mathbb Z$  & $\mathbb Z/ 6 \mathbb Z$ & \multirow{2}{*}{$-7$} & $0\mod 6$ & $p\equiv 1,2,4 \mod 7$\\ \cline{5-6}
  &  & & &$0,4 \mod 6$ & $p\equiv 3,5,6 \mod 7$ \\
   \hline
  2880r6 & $\mathbb Z/ 2 \mathbb Z$  & $\mathbb Z/ 8 \mathbb Z$ & \multirow{2}{*}{$-6$} &$0\mod 8$ &$p\equiv1,5,7,11,19,23\mod 24$\\ \cline{5-6}
   &  & &  & $4 \mod 8$ & $p\equiv 13,17 \mod 24$ \\
   \hline
   150b3 & $\mathbb Z/ 2 \mathbb Z$  & $\mathbb Z/ 10 \mathbb Z$ & \multirow{2}{*}{$5$}& $0 \mod 10$ & $p\equiv 1,4\mod5$\\ \cline{5-6}
   &  &  & & $6 \mod 10$ & $p\equiv 2\mod 5$\\ \cline{5-6}
    &  &  & & $ 8\mod 10$ & $p\equiv 3\mod 5$\\ 
   \hline
   3150bk1 & $\mathbb Z/ 2 \mathbb Z$  & $\mathbb Z/ 16 \mathbb Z$ & \multirow{2}{*}{$-15$} & $0\mod 16$ & $p\equiv 1,2,4,8 \mod 15$  \\ \cline{5-6}
  &  & & & $0,4,8,12 \mod 16$ & $p\equiv 7,11,13,14\mod 15$ \\
   \hline
   14a3 & $\mathbb Z/ 2 \mathbb Z$  & $\mathbb Z/ 2 \mathbb Z\times \mathbb Z/ 2 \mathbb Z$ & \multirow{2}{*}{$-7$} & $0\mod 4$ & $p\equiv 1,2,4 \mod 7$  \\ \cline{5-6}
  &  & & & $0,2 \mod 4$ & $p\equiv 3,5,6 \mod 7$ \\
   \hline
  36a3 & $\mathbb Z/ 2 \mathbb Z$  &$\mathbb Z/ 2 \mathbb Z\times \mathbb Z/ 6 \mathbb Z$& \multirow{2}{*}{$-3$} & $0\mod 12$ & $p\equiv 1 \mod 3$  \\ \cline{5-6}
  &  & & & $0,6 \mod 12$ & $p\equiv 2 \mod 3$ \\
   \hline
   450a3 & $\mathbb Z/ 2 \mathbb Z$  & $\mathbb Z/ 2 \mathbb Z\times \mathbb Z/ 10 \mathbb Z$ & \multirow{5}{*}{$-15$} & $0 \mod 20$ & $p\equiv 1,2,4,8\mod 15$\\ \cline{5-6}
   &  & & & $6,16 \mod 20$ & $p\equiv 7\mod 15$\\ \cline{5-6}
    &  & & & $4,14 \mod 20$ & $p\equiv 11\mod 15$\\ \cline{5-6}
    &  & &  & $8,18 \mod 20$ & $p\equiv 13\mod 5$\\ \cline{5-6}
    &  &  & & $0,10\mod 20$ & $p\equiv 14\mod 15$\\ 
   \hline
   \hline
   50a3 & $\mathbb Z/ 3 \mathbb Z$  & $\mathbb Z/ 15 \mathbb Z$ & \multirow{3}{*}{$5$}& $0 \mod 15$ & $p\equiv 1,4\mod 5$\\ \cline{5-6}
   &  & &  & $3 \mod 15$ & $p\equiv 3\mod 5$\\ \cline{5-6}
   &  & & & $6 \mod 15$ & $p\equiv 2\mod 5$\\
   \hline
  19a1 & $\mathbb Z/ 3 \mathbb Z$ & $\mathbb Z/ 3 \mathbb Z\times \mathbb Z/ 3 \mathbb Z$ & \multirow{2}{*}{$-3$}& $0 \mod 9$ & $p\equiv 1\mod 3$\\ \cline{5-6}
   &  &  & & $0,3,6\mod 9$ & $p\equiv 2\mod 3$\\
   \hline
   \hline
   17a1 & $\mathbb Z/ 4 \mathbb Z$ & $\mathbb Z/ 2 \mathbb Z\times \mathbb Z/ 4 \mathbb Z$ & \multirow{2}{*}{$-1$} & $0 \mod 8$ & $p\equiv 1\mod 4$\\ \cline{5-6}
   &  &  & & $0,4 \mod 8$ & $p\equiv 3\mod 4$\\
   \hline
  192c6 & $\mathbb Z/ 4 \mathbb Z$ & $\mathbb Z/ 2 \mathbb Z\times \mathbb Z/ 8 \mathbb Z$ & \multirow{3}{*}{$-2$} & $0\mod 16$ & $p\equiv 1,3 \mod 8$ \\ \cline{5-6}
  & & & & $4,12\mod 16$ & $p\equiv 5 \mod 8$ \\ \cline{5-6}
   & & & & $0,8\mod 16$ & $p\equiv 7 \mod 8$ \\
   \hline
   \newpage
   \hline
  150c3 & $\mathbb Z/ 4 \mathbb Z$ & $\mathbb Z/ 2 \mathbb Z\times \mathbb Z/ 12 \mathbb Z$ & \multirow{3}{*}{$-15$} & $0\mod 24$ & $p\equiv 1,2,4,8 \mod 15$ \\ \cline{5-6}
  & & & & $0,12\mod 24$ & $p\equiv 11,14 \mod 15$ \\ \cline{5-6}
   & & & & $4,16\mod 24$ & $p\equiv 7,13 \mod 15$ \\
   \hline
   \hline
  50b1 & $\mathbb Z/ 5 \mathbb Z$ &$\mathbb Z/ 15 \mathbb Z$& \multirow{2}{*}{$5$}& $0\mod 15$ & $p\equiv 1,4 \mod 5$ \\ \cline{5-6}
  & & & & $0,10\mod 15$ & $p\equiv 2,3 \mod 5$ \\
   \hline
   \hline
  14a4 & $\mathbb Z/ 6 \mathbb Z$ & $\mathbb Z/ 2 \mathbb Z\times \mathbb Z/ 6 \mathbb Z$ & \multirow{2}{*}{$-7$} & $0\mod 12$ & $p\equiv 1,2,4 \mod 7$ \\ \cline{5-6}
  & & & & $0,6\mod 12$ & $p\equiv 3,5,6  \mod 7$ \\
   \hline
  14a1 & $\mathbb Z/ 6 \mathbb Z$ & $\mathbb Z/ 3 \mathbb Z\times \mathbb Z/ 6 \mathbb Z$ & \multirow{2}{*}{$-3$}& $0\mod 18$ & $p\equiv 1 \mod 3$ \\ \cline{5-6}
  & & & & $0,6,12 \mod 18$ & $p\equiv 2 \mod 3$ \\
   \hline
   \hline
  15a4 & $\mathbb Z/ 8 \mathbb Z$ & $\mathbb Z/ 2 \mathbb Z\times \mathbb Z/ 8 \mathbb Z$ & \multirow{2}{*}{$-1$} & $0\mod 16$ & $p\equiv 1 \mod 4$ \\ \cline{5-6}
  & & & & $0,8 \mod 16$ & $p\equiv 3 \mod 4$ \\
   \hline
   \hline
  63a2 & $\mathbb Z/ 2 \mathbb Z \times \mathbb Z/ 2 \mathbb Z$ & $\mathbb Z/ 2 \mathbb Z\times \mathbb Z/ 8 \mathbb Z$ & \multirow{2}{*}{$-3$} &  $0\mod 16$ & $p\equiv 1 \mod 3$ \\ \cline{5-6}
  & & & & $0,4,8,12 \mod 16$ & $p\equiv 2 \mod 3$ \\
 \hline
  960o6& $\mathbb Z/ 2 \mathbb Z \times \mathbb Z/ 2 \mathbb Z$ &$\mathbb Z/ 2 \mathbb Z\times \mathbb Z/ 12\mathbb Z$ & \multirow{3}{*}{$6$}& $0\mod 24$ & $p\equiv 1,5,19,23 \mod 24$ \\ \cline{5-6}
  & & & & $4,16 \mod 24$ & $p\equiv 7,13 \mod 24$ \\\cline{5-6}
   & & & & $0,12 \mod 24$ & $p\equiv 11,17 \mod 24$ \\
 \hline
 \hline
  15a3 & $\mathbb Z/ 2 \mathbb Z \times \mathbb Z/ 4 \mathbb Z$ &$\mathbb Z/ 2 \mathbb Z\times \mathbb Z/ 8 \mathbb Z$& \multirow{2}{*}{$5$} & $0\mod 16$ & $p\equiv 1,4 \mod 5$ \\ \cline{5-6}
  & & & & $0,8 \mod 16$ & $p\equiv 2,3 \mod 5$ \\ \hline 
 \hline
  90c6 & $\mathbb Z/ 2 \mathbb Z \times \mathbb Z/ 6 \mathbb Z$ & $\mathbb Z/ 2 \mathbb Z\times \mathbb Z/ 12 \mathbb Z$ & \multirow{ 2}{*}{$6$} & $0\mod 24$ & $p\equiv 1,5,19,23 \mod 24$ \\ \cline{5-6}
  & & & & $0,12 \mod 24$ & $p\equiv 7,11,13,17 \mod 24$ \\  
  \hline
\end{longtable}
\end{center}

\end{document}